\documentclass[12pt]{amsart}
\usepackage{latexsym,amscd,amssymb,epsfig} 
\usepackage[dvips]{color}

\theoremstyle{plain}
  \newtheorem{theorem}{Theorem}[section]

\theoremstyle{definition}
  \newtheorem{definition}[theorem]{Definition}

 \theoremstyle{remark}

\numberwithin{equation}{section}

\def\rk{\rm{rk}}

\def\SS{{\mathbb S}}

\def\lk{\rm{lk}}

\begin{document}

\title 
{CW posets after the Poincare Conjecture}

\author{Patricia Hersh}
\address{Box 8205, Department of Mathematics\\
     North Carolina State University\\
      Raleigh, NC 27605}
\email{plhersh@ncsu.edu}

\thanks{The author was supported by NSF grant DMS-1002636.}

\begin{abstract}
Anders Bj\"orner characterized which finite graded partially ordered sets arise
as the closure relation on cells of a finite regular CW complex.  His characterization
of these ``CW posets''
required each open interval $(\hat{0},u)$ 
to have order complex homeomorphic to a sphere of dimension $rk(u)-2$.
Work of Danaraj and Klee showed that sufficient conditions were for the poset to 
be thin and shellable.  The proof of the Poincare Conjecture enables 
the requirement of shellability to be replaced by the homotopy Cohen-Macaulay
property.  This expands the range of tools that 
may be used to prove a poset is a CW poset.
\end{abstract}

\maketitle

\section{Introduction.}

Recall that the order complex $\Delta (P)$ of a finite poset $P$ is the 
simplicial complex whose $i$-faces are exactly the chains 
$u_0 < u_1 < \cdots < u_i$ in $P$.    For any $x<y$ in $P$, denote by
$(x,y)$ the open interval comprised of elements $z\in P$ with $x < z < y$
and let $\Delta (x,y)$ be the order complex for the subposet $(x,y)$ of $P$.

The {\it face poset} $F(K)$ of a simplicial
complex $K$  is the partial order on its faces by inclusion of sets of vertices.
More generally, the {\it closure poset} $F(K)$ of a regular CW complex
$K$ is the
partial order on its cells given by $\sigma \le \tau $ if and only if 
$\sigma \subseteq \overline{\tau }$.  It is easy to see (and well-known)
that  $\Delta ( F(K) \setminus \{ \hat{0} \} )$ is the first barycentric 
subdivision of $K$, hence homeomorphic to $K$.  

Thus, the closure 
of a cell $\tau $ must have $\Delta (F (\overline{\tau } )\setminus \hat{0} )$
homeomorphic to a ball and therefore for each $\sigma < \tau $ in the
closure poset we must have $\Delta (\sigma ,\tau )$ homeomorphic to a 
sphere $\SS^{\dim \tau - \dim \sigma }$.  

Anders Bj\"orner introduced the notion of CW poset in \cite{Bj}:

\begin{definition}
A finite graded poset $P$ is a {\it CW poset} if:
\begin{enumerate}
\item
$P$ has a unique minimal element $\hat{0}$
\item
$P$ has at least one additional element
\item
For each $y$ in $P\setminus \{ \hat{0} \} $, the open interval $(\hat{0},y)$ has order 
complex homeomorphic to a sphere $\SS^{\rk y - 2}$
\end{enumerate}
\end{definition}

He proved the following:

\begin{theorem}[Bj\"orner]
A finite graded poset $P$ is the closure poset of a regular CW complex if and
only if $P$ is a CW poset.
\end{theorem}

Note, for instance, that all 
nontrivial simplicial posets (cf. \cite{St}) with unique minimal element
are  CW posets.
%
%
A necessary condition for a poset to be a CW poset is for each closed 
interval to be Eulerian:

\begin{definition}
A finite graded poset $P$ with unique minimal and maximal elements 
$\hat{0}$ and $\hat{1}$ is {\it Eulerian} if $\mu_P (x,y) = (-1)^{\rk y - \rk x - 2}$
for each $x<y$ in $P$. 
\end{definition}

Necessity is immediate from the fact that a sphere $\SS^d$ has reduced 
Euler characteristic $(-1)^d$, together with the relationship 
$\mu_P(x,y) = \tilde{\chi}(\Delta(x,y) )$
observed by Philip Hall.
A related weaker condition than the Eulerian property is thinness:

\begin{definition}
A finite graded poset $P$ is {\it thin} if each open interval $(x,y)$ for 
$x<y$ with $\rk y - \rk x = 2$ has exactly two elements.
\end{definition}

Not only is thinness of a poset $P$ 
the restriction of the Eulerian property to rank 2 intervals,
but it also implies for graded posets $P$ that $\Delta(P)$ is a pseudomanifold.

\begin{definition}
A simplicial complex is {\it pure} if all maximal faces have the same 
dimension.
\end{definition}

\begin{definition}
A simplicial complex is {\it shellable} if there is a total order on its maximal
faces $F_1,\dots ,F_k$ such that $\overline{F_j}\cap (\cup_{i<j} \overline{F_i})$
is a pure codimension one subcomplex of  $\overline{F_j}$ for each $j\ge 2$.
\end{definition}

A poset is said to be shellable if its order complex is shellable.
A practical set of sufficient conditions for checking that a finite graded poset 
$P$ is a CW poset (cf. \cite{DK}) is as follows:

\begin{theorem}[Danaraj and Klee]
If a finite graded poset $P$ is thin and shellable, then $P$ is a CW poset.
\end{theorem}

For example, in \cite{BW}, Bj\"orner and Wachs proved that Bruhat order
is thin and shellable, hence is a CW poset.

Denote by $\lk_K (\sigma )$ the link of a cell $\sigma $ in a regular CW 
complex $K$.  In the special case of simplicial complexes, 
$\lk_K (\sigma ) = \{ \tau \in K | \tau \cap \sigma = \emptyset, \tau \cup \sigma \in K \} $.

\begin{definition}
A pure $d$-dimensional simplicial complex (or regular CW complex) $K$ 
is homotopy Cohen-Macaulay if 
for each cell $\sigma $ in $K$ (including the empty cell), $lk_K(\sigma )$
is also homotopy equivalent to a wedge of ($d - \dim \sigma -1 )$-dimensional
spheres.
\end{definition}

It is well known that pure, shellable simplicial complexes are homotopy
Cohen-Macaulay.   See \cite{Bj2} for further background.


\section{Thin, homotopy Cohen-Macaulay posets are CW posets}

It seems natural to ask if the shellability requirement may be replaced by 
the more general hypothesis that $P$ is homotopy Cohen-Macaulay.   This
would have the benefit of enabling a wider array of techniques of 
topological combinatorics, for instance discrete Morse theory  (cf. \cite{Fo}, \cite{BH}) and 
the Quillen Fiber Lemma (cf. \cite{Qu}), 
to be used to prove that particular families of posets 
of interest are CW posets.  
%
Interestingly, this follows quite quickly from
the Poincare Conjecture, now a theorem:

\begin{theorem}[Poincare Conjecture]\label{poin-conj}
If $K$ is a manifold that is a simply connected homology sphere, then $K$
is homeomorphic to a sphere.  In other words, if $K$ is a homotopy sphere 
that is also a manifold, then $K$ is homeomorphic to a sphere.
\end{theorem}

Theorem ~\ref{poin-conj}
 follows in dimension two from the classification of surfaces.  It was proven
in dimension at least 5 by Smale in \cite{Sm} (see also \cite{Ze}), 
in dimension 4 by
Freedman in \cite{Fr}, and finally in dimension 3 by Perelman (see \cite{Pe}, \cite{Pe2},
\cite{Pe3}, with full detail  in \cite{KL}).
Now to a consequence for topological combinatorics.
 
\begin{theorem}
Let $P$ be a finite graded poset with unique minimal element $\hat{0}$ and
at least one other element.  Then $P$ is a CW poset if $P$ is thin
and homotopy Cohen-Macaulay.
\end{theorem}

\begin{proof}
First note that the theorem is true for $P$ of rank 2, in which case $P$ has
exactly 4 elements and exactly one nontrivial open interval. 
Now suppose by induction the theorem holds
for $P$ of rank strictly less than $r$.  Consider $P$ of exactly rank $r$, and 
consider any open interval $(x,y)$ in $P$.
The link of any vertex in $\Delta(x,y)$  is the order complex of a 
product of open intervals of strictly smaller rank.  By induction, it is therefore
the join of simplicial complexes each homeomorphic to a sphere, therefore
is itself homeomorphic to a sphere.  This implies that $\Delta(x,y)$ is a 
manifold.  Since $P$ is homotopy Cohen-Macaulay, we also have that 
$\Delta(x,y)$ is homotopy equivalent to a wedge of spheres; by thinness, we
may assume $\rk y - \rk x > 2$, implying $\Delta(x,y)$ is connected.  But the 
top homology group over the integers of any connected manifold is either 
the integers or 0, depending whether the manifold is orientable or not.  Thus,
$\Delta(x,y)$ is either contractible or homotopy equivalent to a  sphere.
In the former case, $\Delta(x,y)$ would be a manifold with boundary, with
this boundary  pure of codimension one; this contradicts thinness,
which implies each codimension one face is contained in precisely two facets.
Thus, $\Delta(x,y)$ is homotopy equivalent to a sphere.  But the Poincare 
Conjecture yields that any 
manifold that is homotopy equivalent to a sphere 
must be homeomorphic to a sphere.
\end{proof}

Anders Bj\"orner has independently obtained related results [Bj3], also using the
Poincare Conjecture.
It seems natural now to ask whether the Poincare Conjecture has further 
ramifications of a similar spirit for topological combinatorics.

\section{Acknowledgments}

The author thanks Anders Bj\"orner and Lenny Ng for helpful comments.

\end{document}